\documentclass[11pt,reqno]{amsart}
\usepackage{amsmath, amssymb, amsthm}
\usepackage{graphicx}
\usepackage{xcolor}
\usepackage{url}
\usepackage[breaklinks]{hyperref}

\renewcommand{\Im}{\operatorname{Im}}

\newcommand{\scrF}{\mathcal F}
\newcommand{\scrL}{\mathcal L}
\newcommand{\C}{\mathbb C}
\newcommand{\R}{\mathbb R}

\newcommand{\SL}{\operatorname{SL}}

\renewcommand{\Re}{\operatorname{Re}}
\renewcommand{\Im}{\operatorname{Im}}

\newcommand{\cL}{{\mathcal L}}

\renewcommand{\(}{\left\(}
\renewcommand{\)}{\right\)}
\renewcommand{\[}{\left\[}
\renewcommand{\]}{\right\]}
\numberwithin{equation}{section}
 \theoremstyle{plain}
\newtheorem{theorem}{Theorem}[section]
\newtheorem{lemma}[theorem]{Lemma}

\newtheorem{defn}[theorem]{Definition}
\newtheorem{corollary}[theorem]{Corollary}

\newtheorem{proposition}[theorem]{Proposition}

   \makeatletter
\def\proof{\@ifnextchar[{\@oproof}{\@nproof}}
\def\@oproof[#1][#2]{\trivlist\item[\hskip\labelsep\textit{#2 Proof of\
#1.}~]\ignorespaces}
\def\@nproof{\trivlist\item[\hskip\labelsep\textit{Proof.}~]\ignorespaces}

\makeatother

\title{Summation Formulas for Harmonic Maass Forms}
\date{}
\author{Nikolaos Diamantis}
\address{University of Nottingham}
\email{nikolaos.diamantis@nottingham.ac.uk}
\author{Joshua Pimm}
\address{University of Nottingham}
\email{joshua.pimm@nottingham.ac.uk}
\begin{document}

\begin{abstract}
In a previous work \cite{BDGRT}, a summation formula for harmonic Maass forms of polynomial growth was established. In this note, we use the theory of $L$-series of harmonic Maass forms to state and prove a summation formula for such forms without any restrictions on their growth. We deduce a summation formula for the partition function. We further employ the same theory to derive a result on classical modular forms, namely, a summation formula and the asymptotics of a Riesz sum attached to a cusp form. 
\end{abstract}
\maketitle
\begin{center}

\end{center}


\section{Introduction}
Summation formulas and related analytic techniques have not featured prominently in the theory of harmonic Maass forms, which is surprising given the important role such techniques have played in the theory of classical modular forms and their applications. Part of the reason for this absence was the lack of a systematic theory of $L$-series attached to harmonic Maass forms. However, such a theory is now available (see \cite{DLRR23}) and several foundational aspects of it have been explored. 

An application based on a summation formula was already given in the work \cite{DLRR23} where the $L$-series for harmonic Maass forms was introduced: it represented a contribution to the ``harmonic lift" problem, via a summation formula that sheds light on the behaviour of the “holomorphic part” of a harmonic Maass form with a given ``shadow". More recently, a summation formula for a certain Riesz means \cite{BDGRT} was established, but only for harmonic Maass forms of polynomial growth. That was enough to obtain arithmetic applications for Hurwitz class numbers and for Fourier coefficients of negative half-integral weight Eisenstein series, but the class of harmonic Maass forms of polynomial growth is quite restricted. In this note, we establish a summation formula for harmonic Maass form without any constraints on the growth. This will be stated and proved in general form in Theorem \ref{SF1}, but here we present only a special case to give the general idea. Specifically, let, for some $n_0 \in \mathbb N_0$,
$$f(z)=\sum_{n \ge -n_0} a(n) e^{2 \pi i n z}$$
be a weakly holomorphic modular form of weight $k \in \frac{1}{2}\mathbb N$ for $\Gamma_0(N)$. Let $b(n)$ be the $n$-th Fourier coefficient of $g(z):= (\sqrt{N} \, z)^{-k} f(-1/N\, z)$ and suppose that $a(n), b(n)=O(e^{C_f \sqrt{n}})$, as $n \to \infty$, for some $C_f>0$. We then have the following.
\begin{theorem} For $C>\max(C^2_f\sqrt{N}/(8 \pi ), 2 \pi n_0/\sqrt{N})$ and for all $X>0$, we have
\begin{multline*} \hskip -2mm \sum_{n \ge -n_0} a(n) K_0 \left (2 \sqrt{D_n(C+\sqrt{N} X)} \right )\\=
i^kC^{\frac{k}{2}}\sum_{n \ge -n_0} b(n) \left ( \sqrt{N}X+D_n \right )^{-\frac{k}{2}} K_k \left ( 2 \sqrt{C ( D_n+\sqrt{N}X )} \right)
\end{multline*}
where $D_n:=C+2 \pi n/\sqrt{N}$ and $K_s(z)$ is the modified Bessel function
\end{theorem}
An immediate application of this is a summation formula for partition functions, which is stated in Corollary \ref{p(n)}.

A second feature of the summation formula proved in \cite{BDGRT} is that it did not require the full use of the $L$-series defined in \cite{DLRR23}. Since it involved only polynomial growth harmonic Maass forms, the $L$-series attached to the specific type of harmonic Maass forms in \cite{ShanSingh22} sufficed. A natural question then is what kind of summation formulas for Riesz means we derive if we start from the more general $L$-series of \cite{DLRR23} having test functions as arguments. That means exploiting the freedom of selecting a suitable family of test functions, and in this note, we apply this principle to deduce a summation formula, a special case of which is the following:
\begin{theorem}\label{Sect4int} Let $f$ be a holomorphic cusp form of weight $k \in 2\mathbb N$ for $\SL_2(\mathbb Z)$ with Fourier coefficients $c(n)$. Then, for $\rho> 1,$ and each $X>0$ we have
\begin{multline}\label{TH4int}
\sum_{n<X/(2 \pi)}\frac{c(n)e^{-2 \pi n}}{(2\pi n)^{\frac{k-1}{2}}} \left (X-2 \pi n \right )^{\rho}= i^{-k} X^{\rho}L_f\left (\psi_{\frac{k-1}{2}} \right )
\\+\frac{\Gamma(\rho+1)}{2 \pi i^{k+1}\Gamma(k)} \sum_{n>0}\frac{c(n)}{e^{2 \pi n}}
 \int_{(-\varepsilon)}X^{s+\rho}\frac{\Gamma \left (\frac{k+1}{2}-s \right )\Gamma(s)}{\Gamma(s+\rho+1)}M \left (\frac{k-1}{2}+s, k, 2 \pi n \right ) ds
\end{multline}
where $M(a,b,z)$ is Kummer's confluent hypergeometric function.
\end{theorem}
The value $L_f\left (\psi_{\frac{k-1}{2}} \right )$ on the right-hand side of \eqref{TH4int} originates in the functional equation of $L_f(\varphi)$ in the same way as the central critical value of the standard $L$-series of a holomorphic modular form. It would be interesting to see if there is a more intrinsic characterisation of such ``critical values" of $L_f(\varphi)$. 

As an application of Theorem \ref{Sect4int}, we deduce asymptotics for the weighted moment in the LHS of \eqref{TH4int}:
\begin{corollary}
    With the assumptions and notation of Theorem \ref{Sect4int}, we have, for $\rho> 1$ and each $X>0$,
\begin{equation}
\sum_{n<X/(2 \pi)}\frac{c(n)e^{-2 \pi n}}{(2\pi n)^{\frac{k-1}{2}}} \left (1-\frac{2 \pi n}{X} \right )^{\rho}= i^{-k} L\left (\psi_{\frac{k-1}{2}} \right )+O(X^{-\varepsilon}).
\end{equation}
\end{corollary}

In the next section, we present the basic objects we will be studying. In Section \ref{section3}, we state and prove the summation formula for general harmonic Maass forms. In the final section, we prove the Riesz-type summation formula and use it to derive the application to weighted moments of Fourier coefficients of holomorphic cusp forms. 

\section{Preliminaries}
Let $k \in \frac{1}{2} \mathbb{Z}$ and $N \in \mathbb N.$ We write an element of the upper half-plane $\mathbb H$ as $z = x+iy$, with $x,y  \in \mathbb{R}$, $y>0$. Further $q:=e^{2 \pi i \, z}.$ We choose the principal branch of the square-root throughout.
  
The {\it (Petersson) slash operator} is defined as 
    \begin{align*}
    &\left(f\vert_k\gamma\right)(\, z) := \begin{cases}
    (c\, z+d)^{-k} f(\gamma\, z) & \text{if } k \in \mathbb{Z}, \\
    \left(\frac{c}{d}\right)\varepsilon_d^{2k}(c\, z+d)^{-k} f(\gamma\, z) & \text{if } k \in \frac{1}{2}+\mathbb{Z},
    \end{cases} 
    \quad \\
    & \text{for} \, \, \gamma = \left(\begin{matrix} a & b \\ c& d \end{matrix}\right) \in \begin{cases}
    \operatorname{SL}_2(\mathbb{Z}) & \text{if } k \in \mathbb{Z}, \\
    \Gamma_0(4) & k \in \frac{1}{2}+\mathbb{Z},
    \end{cases} 
    \end{align*}
    where $\left(\frac{c}{d}\right)$ denotes the Kronecker symbol, and
    \begin{align*}
    \varepsilon_d := 
    \begin{cases}
1 & \text{if } d \equiv 1 \pmod{4} , \\
i & \text{if } d \equiv 3 \pmod{4}.
\end{cases}
    \end{align*}
We define the action of $w_N=\left ( \begin{smallmatrix} 0 & -1/\sqrt{N} \\ \sqrt{N} & 0 \end{smallmatrix} \right )$ on a function $f: \mathbb H \to \mathbb C$ as
    \begin{equation}\label{eq:fricke}
      f|_k w_N (\, z) :=  N^{k/2} (N \, z)^{-k} f(-1/N\, z). 
    \end{equation}
For $\varphi: \mathbb R_+ \to \mathbb C,$ we also define the function $\varphi|_kw_N$ given by
  \begin{equation}\label{eq:fricketest}
      \varphi|_k w_N (x) := (N x)^{-k} f(1/Nx) \qquad \text{for $x \in \mathbb R_+$} 
    \end{equation}
and, for $s \in \mathbb C$, we set 
\begin{equation}
\varphi_s(x):=\varphi(x)x^{k-1}.
\end{equation}

    \begin{defn} Let $k \in \frac{1}{2}\mathbb Z$ and $N \in \mathbb N$ (with $4|N$ whenever $k \in \frac{1}{2} + \mathbb{Z}$). 
Let $\chi$ be a Dirichet character $\mod N.$ We call a smooth function $f\colon \mathbb{H} \to \mathbb{C}$ a {\it harmonic Maass form of weight $k$ and character $\chi$} on $\Gamma_0(N)$, if it satisfies the following conditions:
    \begin{enumerate}
    \item For all $\gamma \in \Gamma_0(N)$, we have $f|_k \gamma = \chi(\gamma)f$.
    \item We have $\Delta_k f = 0$, where $
    \Delta_k := -y^2\left(\frac{\partial^2}{\partial x^2}+\frac{\partial^2}{\partial y^2}\right) + iky\left(\frac{\partial}{\partial x} + i\frac{\partial}{\partial y}\right).$
    \item For each $\sigma \in \SL_2(\mathbb Z)$ , there is a polynomial $P_{\sigma}$ such that $(f|_k \sigma)(z)-P_{\sigma}(q^{-1})=O(e^{-\varepsilon_{\sigma} y})$ as $y \to \infty$, for some $\varepsilon_{\sigma}>0$. 
    \end{enumerate} \label{def:hmfpg}
    \end{defn}
    We let $H_k(N,\chi)$ denote the space of such functions.  Each $f \in H_k(N,\chi)$ has a Fourier expansion of the form
 \begin{equation}\label{eq:FourierExpansion}
       f(\, z) = \sum_{n \ge -n_0}^{\infty} c^+(n) q^n + \sum_{n=1}^{\infty} c^-(n)\Gamma(1-k, 4 \pi n y) q^{-n}. 
       \end{equation}
and similar expansions at all other cusps. For some $C>0,$ we have $c^+(n)=O(e^{C\sqrt{n}})$ as $n \to \infty$, while $c^-(n)$ have polynomial growth as $n \to \infty$ (see \cite[Sec. 5.2]{thebook}) 
    
If $c^-(n)=0$ for all $n > 0$ and at all cusps, then $f$ is called a {\it weakly holomorphic modular form.} Their space is denoted by $M^!_k(N, \chi)$. If, further $c^+(0)=0$ (at all cusps), $f$ is called a {\it weakly holomorphic cusp form.} Their space is denoted by $S^!_k(N, \chi)$. 
We retrieve the space of weight $k$ holomorphic modular (resp. cusp) forms for level $N$ as the space $M_k(N,\chi)$ (resp. $S_k(N,\chi)$) of $f \in M^!_k(N,\chi)$ such that $c^+(n)=0$ for all $n<0$ (resp. $n \le 0$) (and similarly at the other cusps).

Another important subspace of $H_k(N, \chi)$ is the space $H'_k(N,\chi)$ of harmonic Maass forms such that for each $\sigma \in \SL_2(\mathbb Z)$, $(f|_k \sigma)(z)-a_{\sigma}=O(e^{-\varepsilon_{\sigma} y})$ as $y \to \infty$, for some $a_{\sigma} \in \mathbb C$ and $\varepsilon_{\sigma}>0$. 

The spaces $H_{k}(N, \chi)$ and $S_{2-k}(N, \bar \chi)$ are connected by the operator $\xi_k$ mapping a $f \in H_{k}(N, \chi)$ with expansion \eqref{eq:FourierExpansion} to its ``shadow''
\begin{equation}\label{xi} 
\xi_{k}f:=2iy^{k} \overline{\frac{\partial f}{\partial \bar z}} =
 -\sum_{n=1}^{\infty} \overline{c^-(n)}(4 \pi n)^{1-k}q^{n}
       \in S_{2-k}(N, \bar \chi).
\end{equation}
This, in particular, implies that
\begin{equation}\label{genuine} 
H_{k}(N, \chi)=M^!_k(N, \chi), \qquad \text{if $k>2$.}
\end{equation}

We now recall the definition of $L$-series of $f \in H_k(N, \chi)$ given in \cite{DLRR23}. The domain $\scrF_f$ we will be using in this is note consists of $\varphi\in C(\R, \C)$ such that the Laplace integral $$(\scrL\varphi)(s)=\int_0^{\infty}\varphi(x)e^{-sx}dx$$ (resp. $(\scrL\varphi_{2-k})(s)$) converges absolutely for all $s$ with $\Re(s) \ge -2 \pi n_0$ (resp. $\Re(s)>0$), 
and the following series converges:  
\begin{equation}\label{Ff}
\sum_{\substack{n \ge -n_0}} |c^+(n)| (\scrL|\varphi|)\left (2 \pi n\right ) 
+ \sum_{n>0} |c^-(n)|\left (4\pi n\right )^{1-k}\int_0^{\infty}\frac{(\scrL |\varphi_{2-k}|)\left (2\pi n(2t+1)\right )}{(1+t)^{k}}dt. 
\end{equation}
The value of the series $L_f$ at $\varphi \in \scrF_f$ is then 
\begin{equation*}
L_f(\varphi)=L^1_{f}(\varphi)+ L^2_{f}(\varphi), \quad \text{where}
\end{equation*}
\begin{equation}\label{L1L2}
L^1_{f}(\varphi)=\sum_{\substack{n \ge -n_0}} c^+(n) (\scrL \varphi)\left (2 \pi n\right ) \, \text{and} \, 
L^2_{f}(\varphi)=\sum_{n>0} \frac{c^-(n)}{\left (4\pi n\right )^{k-1}}\int_0^{\infty}\frac{(\scrL \varphi_{2-k})\left (2\pi n(2t+1)\right )}{(1+t)^{k}}dt.
\end{equation}
A useful identity is \cite[(4.14)]{DLRR23}
\begin{equation}\label{nonhol}
\int_0^{\infty}\Gamma\left(1-k, 4 \pi n y \right)e^{2 \pi n y}\varphi(y)dy
= \left (4\pi n \right )^{1-k}\int_0^{\infty}\frac{(\scrL\varphi_{2-k})
 \left (2\pi n(2t+1)\right )}{(1+t)^k} dt. 
\end{equation}

\bigskip 

\section{A summation formula for a weakly holomorphic cusp forms}\label{section3}
Let $f \in H_k(N, \chi)$ with Fourier expansion \eqref{eq:FourierExpansion}.
We let $d^{\pm}(n)$ be the coefficients of the corresponding expansion of $g:=f|w_N \in H_k(N, \bar \chi)$ (or $H_k \left (N, \bar \chi \left ( \frac{N}{\bullet}\right ) \right )$, if $k \in \frac12+\mathbb Z$ ). We will first prove a ``pre-summation formula" for broad pair of test functions. We recall the notation 
$$\mathcal M(f)(s)=\int_0^{\infty}f(x)x^{s-1} dx$$
for the Mellin transform of $f:\mathbb R_+ \to \mathbb C.$
\begin{proposition}
    \label{preSF} Fix a continuous $F: \mathbb R_+ \to \mathbb C$ such that $\mathcal M(F)(s)$ converges absolutely and uniformly in compacta in $\{s; \Re(s)<0\}$ giving an absolutely integrable function over each $-\varepsilon+i\mathbb R$. Let $\varphi \in C(\mathbb R_+, \mathbb C)$ such that, for all $s \in \mathbb C,$ $\varphi_s \in \scrF_f, 
    \varphi_s|_{2-k}w_N \in \scrF_g$ and $L_f(\varphi_s), L_g(\varphi_s|_{2-k}w_N)$ are entire functions of $s.$ 
    We then have
\begin{multline}
        \label{preSFform} \sum_{n \ge -n_0} c^+(n) \int_0^{\infty}F(x)\varphi(x) e^{-2 \pi n x}\frac{dx}{x} 
        +
        \sum_{n >0} c^-(n)\int_0^{\infty}F(x) \varphi(x) e^{2 \pi n x} \Gamma(1-k, 4 \pi n x) \frac{dx}{x}
        \\=i^k N^{-\frac{k}{2}} \left ( \sum_{n \ge -n_0} d^+(n) \int_0^{\infty} F(x) \varphi(x)x^{-1-k} e^{-\frac{2 \pi n }{Nx}} dx \right.   \\
         \left. +\sum_{n >0} d^-(n)\int_0^{\infty}F(x) \varphi(x)x^{-k-1} e^{\frac{2 \pi n }{Nx}} \Gamma \left (1-k, \frac{4 \pi n}{Nx} \right ) dx
 \right )
    \end{multline}
\end{proposition}
\begin{proof} Set $\psi(s)=\mathcal M(F(x^{-1}))(s)=\mathcal M(F)(-s).$ By the assumptions on $F$ this is analytic for $\Re(s)>0$ and $F(x)=\mathcal M^{-1}(\psi)(x^{-1})$. For some $\varepsilon>0$, we consider \begin{equation} \label{genI}
       I:=I_1+I_2 \quad \text{where $I_j:=\frac{1}{2 \pi i} \int_{(\varepsilon)} \psi(s)L^j_f(\varphi_s)ds.$}
    \end{equation}
    Because of the assumptions on $\varphi_s$ and $F$, we see that
    \begin{multline}\label{crit}\int_{(\varepsilon)}|\psi(s)|\left ( \sum_{n \ge -n_0}|c^+(n)| \int_0^{\infty}|\varphi_s(x)|e^{-2 \pi n x}dx \right. \\
    \left. +\sum_{n >0}|c^-(n)| \int_0^{\infty}|\Gamma(1-k, 4 \pi n x) \varphi_s(x)|e^{2 \pi n x}dx \right ) ds< \infty \end{multline}
   and hence, we can interchange all summations and integrations. Unraveling the formula for $L_f(\varphi_s)$
   we obtain
   $$I_1=\frac{1}{2 \pi i}\sum_{n \ge -n_0} c^+(n) \int_0^{\infty}\varphi(x)x^{-1} e^{-2 \pi n x} \int_{(\varepsilon)}\psi(x)x^{s}ds dx$$
   and
    $$I_2=\frac{1}{2 \pi i}\sum_{n >0} c^-(n) \int_0^{\infty}\Gamma(1-k, 4 \pi n x)\varphi(x)x^{-1} e^{2 \pi n x} \int_{(\varepsilon)}\psi(x)x^{s}ds dx.$$
    By the assumptions on $\varphi_s$ and $\varphi_s|_{2-k} w_N$, Theorem 4.5 of \cite{DLRR23} holds to give 
    \begin{equation}\label{FE}
    L_f(\varphi_s)=i^k N^{1-\frac{k}{2}}L_g(\varphi_{s}|_{2-k}w_N),
    \end{equation}
    and hence
    \begin{equation}
        \label{comp}
I=i^kN^{1-\frac{k}{2}} \frac{1}{2 \pi i} \int_{(\varepsilon)} \psi(s)L^j_g(\varphi_s|_{2-k}w_N)ds.    \end{equation}
To compute the RHS of \eqref{comp} we first note, with the change of variables $1/(Nx) \to x$, that
$$\int_{(\varepsilon)}\psi(s) \int_0^{\infty}\varphi(1/(Nx))(xN)^{k-s-1}e^{-2 \pi n x}dx ds=
\frac{1}{N}\int_{(\varepsilon)}\psi(s) \int_0^{\infty}\varphi(x)x^{s-k-1}e^{-\frac{2 \pi n}{Nx}}dx ds
$$
from which the first term in the RHS of \eqref{preSFform} is obtained after changing the order integration.
The second term is computed similarly. The proposition follows by replacing these terms and $I_1, I_2$ into \eqref{comp}.
\end{proof}
We will deduce our summation formula by applying this proposition with a specific $\psi$ and $\phi$. Specifically,
assume that $c^{+}(n), d^{+}(n)=O(e^{C_f \sqrt{n}})$, as $n \to \infty$. Fix some constant $C>\max(C^2_f\sqrt{N}/(8 \pi ), 2 \pi n_0/\sqrt{N})$ and let $\varphi: \mathbb R_+ \to \mathbb R$ be given by
$$\varphi(x)=e^{-C \left ( \sqrt{N}x+\frac{1}{ \sqrt{N}x}\right )}, \qquad \text{for $x>0$.}$$
Further, for each $Y>0$, we set $$F(x):=e^{-\frac{1}{xY}}.$$
We first observe that, for $k \in \frac12 \mathbb Z$, we have
\begin{equation}\label{inv} \varphi_s|_{2-k}w_N=N^{k-s-1}\varphi_{k-s}.
\end{equation}
We also have the following lemma 
\begin{lemma}\label{FfFg}For each $s \in \mathbb C$, $\varphi_s \in \mathcal F_f$ and $\varphi_s|_{2-k}w_N \in \mathcal F_g$. Further, viewed as functions of $s \in \mathbb C$, $L_f(\varphi_s)$ and $L_{f|_k w_N}(\varphi_s|_{2-k}w_N)$ are entire.
\end{lemma}
\begin{proof} Since $C> 2 \pi n_0/\sqrt{N}$, the integral defining $\cL \varphi_s(2\pi n)$ is absolutely convergent for $n\geq -2 \pi n_0$. Further, for $n>0$,
\begin{equation}\label{Lap} \cL \varphi_s(2 \pi n)=\int_{0}^{\infty} e^{-(C\sqrt{N}+2 \pi n)x-\frac{C}{\sqrt{N}x}}x^{s-1}dx=2 \left ( \frac{C/\sqrt{N}}{C\sqrt{N}+2 \pi n} \right )^{s/2}K_{s}\left(2\sqrt{C\left(C+\frac{2\pi n}{\sqrt N}\right)}\right)
\end{equation}
by \cite[(10.32.10)]{DLMF}. Here $K_s(z)$ denotes the modified Bessel function.  With \cite[(10.40.2)]{DLMF}, we deduce  
$$ \cL \varphi_s(2 \pi n)=O \left (\left (C+\frac{2 \pi n}{\sqrt N} \right )^{-\frac{2\Re(s)+1}{4}}e^{-2\sqrt{C \left (C+\frac{2 \pi n}{\sqrt{N}}\right )}}\right ) \quad \text{as $n\rightarrow\infty$.} $$
Since $C> C^2_f\sqrt{N}/(8 \pi )$, this implies that the first sum of \eqref{Ff}
converges uniformly for $s$ in compacta. For the non-holomorphic part, the same combination of \cite[(10.32.10)]{DLMF} and \cite[(10.40.2)]{DLMF}, together with the asymptotics of $\Gamma(1-k, x)$, implies that, for some $A \in \mathbb R$ and $B>0$ that depend only on $\Re(s)$, we have 
\begin{equation*}
\int_0^{\infty}\Gamma\left(1-k, 4 \pi n y \right)e^{2 \pi n y}|\varphi_s(y)|dy
\ll \int_0^{\infty}e^{-(2 \pi n+C\sqrt{N})y-\frac{C}{\sqrt{N}y}}y^{A}dy \ll e^{-B \sqrt{n}}.
\end{equation*}
With \eqref{nonhol} and the polynomial growth of $c^-(n)$, we deduce that the second sum of \eqref{Ff} converges uniformly too. Therefore, $\varphi_s \in \mathcal F_f$ and $L_f(\varphi_s)$ is holomorphic in $s.$ 
With \eqref{inv}, it is easy to see the analogous assertions for $\varphi_s|_{2-k}w_N$ and $L_{f|_k w_N}(\varphi_s|_{2-k}w_N).$
\end{proof}
With this notation and facts, we are now ready to state and prove our summation formula.
\begin{theorem}\label{SF1} Suppose that $k$ is non-positive integer. With the above notation, set $D_n:=C+2 \pi n/\sqrt{N}.$ 
Then we have, for all $X>0$,
\begin{multline}\label{SF} \sum_{n \ge -n_0} c^+(n) K_0 \left (2 \sqrt{D_n(C+\sqrt{N} X)} \right )\\+
\sum_{\substack{\ell=0}}^{|k|} \frac{|k|!}{\ell!} \left ( \frac{4 \pi \sqrt{C}}{\sqrt{N}}\sqrt{1+\frac{\sqrt{N}X}{C}}\right )^{\ell} \sum_{n>0}c^-(n) \left ( \frac{n}{\sqrt{D_n}}\right )^{\ell} K_{\ell} \left(2 \sqrt{D_n(C+\sqrt{N}X)} \right )
\\=
i^k C^\frac{k}{2}\sum_{n \ge -n_0} d^+(n) \left ( \sqrt{N}X+D_n \right )^{-\frac{k}{2}} K_k \left ( 2 \sqrt{C ( D_n+\sqrt{N}X )} \right)\\
+i^k C^{\frac{k}{2}} \sum_{\ell=0}^{|k|}\frac{|k|!}{\ell!} \left ( \frac{4 \pi \sqrt{C}}{\sqrt{N}}\right )^\ell
\sum_{n>0}d^-(n)n^{\ell}
\left ( D_n+\sqrt{N}X\right )^{-\frac{\ell+k}{2}} K_{\ell+k} \left ( 2\sqrt{C (D_n+ \sqrt{N} X)}\right ).
\end{multline}
In the special case that $f$ is weakly holomorphic, \eqref{SF} holds for all $k \in \frac{1}{2}\mathbb Z$.
\end{theorem}
\begin{proof} Set $\psi(s)=\mathcal M(F(x^{-1}))(s)=Y^s\Gamma(s).$
By Lemma \ref{FfFg} and the absolute integrability of $\psi$ due to Stirling's formula, we see that Proposition \ref{preSF} holds. We will compute each of the integrals appearing in \eqref{preSFform}. 
For our $\varphi$ and $F$, we have, with \cite[(10.32.10)]{DLMF}, 
\begin{multline}
    \label{first}
    \int_0^{\infty}\varphi(x)x^{-1} e^{-2 \pi n x} F(x)dx=
        \int_0^{\infty} e^{-\left ( (C\sqrt{N}+2 \pi n)x+\frac{\left ( \frac{1}{Y}+\frac{C}{\sqrt{N}}\right )}{x} \right )} \frac{dx}{x}\\=2 K_0 \left (2 \sqrt{D_n \left ( \frac{\sqrt{N}}{Y}+C \right )} \right ).
\end{multline}
For the first integral in the RHS of \eqref{preSFform} we have, 
\begin{multline}
    \label{third}
    \int_0^{\infty} e^{-C \left (\sqrt{N}x+\frac{1}{\sqrt{N}x} \right )-\frac{2 \pi n}{Nx}-\frac{1}{xY}} x^{-k} \frac{dx}{x}\\
        =2  (NC)^{\frac{k}{2}} \left ( C+\frac{2 \pi n}{\sqrt{N}}+\frac{\sqrt{N}}{Y}\right )^{-\frac{k}{2}}K_{-k} \left (2 \sqrt{C \left (D_n+\frac{\sqrt{N}}{Y} \right )} \right ).
\end{multline}
For the remaining two integrals, which do not appear in the case of weakly holomorphic forms, we restrict to $k$ negative integer in order to obtain simpler formulas. This is thanks to the following identity, which is valid for $k \in -\mathbb N_0$
\begin{equation}
    \label{Gamma}\Gamma(1-k, x)=|k|!e^{-x} \sum_{\ell=0}^{|k|}\frac{x^{\ell}}{\ell !}
\end{equation}
Then, the second integral of \eqref{preSFform} becomes 
 \begin{multline}\label{second}
 |k|!\sum_{\substack{\ell=0}}^{|k|}  \frac{(4 \pi n)^\ell}{\ell !} \int_0^{\infty} e^{-(2\pi n+C\sqrt{N})x -\frac{\frac{C}{\sqrt{N}}+\frac{1}{Y}}{x}}x^{\ell}\frac{dx}{x} \\
 =\sum_{\substack{\ell=0}}^{|k|} \frac{2|k|!}{\ell!} \left ( \frac{4 \pi \sqrt{C}}{\sqrt{N}} \sqrt{1+\frac{\sqrt{N}}{CY}}\right )^{\ell} \left ( \frac{n}{\sqrt{D_n}}\right )^{\ell} K_{\ell} \left(2 \sqrt{C D_n+\frac{\sqrt{N} D_n}{Y}} \right )
 \end{multline}
The fourth integral of \eqref{preSFform}, after the change of variables $x \to 1/(Nx)$ becomes
$$ \int_0^{\infty} e^{-\frac{Nx}{Y}-C \left ( \sqrt{N} x+ \frac{1}{\sqrt{N} x} \right )+2\pi nx}(Nx)^{k}\Gamma(1-k, 4 \pi n x) \frac{dx}{x}$$ 
which, with \eqref{Gamma} and \cite[(10.32.10)]{DLMF}, gives
 \begin{equation}\label{fourth} (CN)^{\frac{k}{2}} \sum_{\ell=0}^{|k|}\frac{2|k|!}{\ell!} \left ( \frac{4 \pi \sqrt{C}}{\sqrt{N}}\right )^\ell
  n^{\ell}
\left ( D_n+\frac{\sqrt{N}}{Y}\right )^{-\frac{\ell+k}{2}} K_{\ell+k} \left ( 2\sqrt{C \left (D_n+ \frac{\sqrt{N}}{Y} \right )} \right )
 \end{equation}
Substituting \eqref{first}, \eqref{second}, \eqref{third} and \eqref{fourth} into \eqref{preSFform} implies the theorem, upon setting $X=1/Y$.
\end{proof}
{\bf Remark.} The conditions on $k$ in the non-weakly holomorphic case may seem quite restrictive, but, if $k>2$, $f$ will be weakly holomorphic anyway by \eqref{genuine}. The half-integral weights $<2$ are excluded from the statement only because the part of the summation formula corresponding to the non-holomorphic of $f$ would give more complicated expressions involving the error function. Otherwise, a summation formula is also possible in that case.

There is a simple application of Theorem \ref{SF1} to partition numbers $p(n)$. It is well-known that, if $\eta(z)$ denotes the Dedekind eta function, we have
$$\frac{1}{\eta(z)}=q^{-\frac{1}{24}}\prod_{n \ge 1} (1-q^n)^{-1}=q^{-\frac{1}{24}}\sum_{n \ge 0}p(n)q^n.$$
On the other hand $\eta(24z) \in S_{\frac{1}{2}}(576, \chi_{12})$, where $\chi_{12}(n)=\left ( \frac{12}{n}\right )$ (e.g. Cor. 1.62 of \cite{ono}). We can thus apply Theorem \ref{SF1} to 
$$f(z)=\frac{1}{\eta(24z)}=\sum_{n \ge 0}p(n) q^{24n-1} \in M^!_{-\frac12}(576, \chi_{12}).$$
By the transformation law of $\eta(z)$ we deduce that
$$g(z)=(f|_{-\frac{1}{2}}w_{576})(z)=\frac{1+i}{\sqrt{2}}f.$$
Hence we have $n_0=1$, $c^+(n)=p((n+1)/24)$ if $24|(n+1)$ and $0$ otherwise, and $d^+(n)=(1+i)c^+(n)/\sqrt{2}$. By the asymptotics for $p(n)$ (\cite[Theorem 5.5]{ono}), we have $$c^+(n), d^+(n) \ll e^{\pi \sqrt{\frac{2}{3}}\sqrt{\frac{n+1}{24}}} \ll e^{\frac{\pi}{6} \sqrt{n}}$$
and hence $C_f$ can be taken to be $\pi/6$. Therefore, we can choose
$$C=\max \left ( \frac{\pi^2}{36}\frac{\sqrt{576}}{8 \pi}, \frac{2\pi}{\sqrt{576}} \right )+\varepsilon=\frac{\pi}{12}+\varepsilon.$$
Then, in particular, $D_{24n-1}=2 \pi n+\varepsilon$. Finally, we note that $K_{-1/2}(z)=\sqrt{\pi/(2z)}e^{-z}$ \cite[10.39.2]{DLMF}. From these remarks, we deduce the following
\begin{corollary}\label{p(n)}
    For all $X>0$, we have
    \begin{multline*} \sum_{n \ge 0} p(n) K_0 \left (2 \sqrt{(2 \pi n+\varepsilon)\left (\frac{\pi}{12}+\varepsilon+24X \right )} \right )\\=
\frac{\sqrt{\pi}}{2} \left ( \frac{\pi}{12}+\varepsilon\right )^\frac{-1}{2}
\sum_{n \ge 0} p(n) 
\left ( 24X+2 \pi n+\varepsilon \right )^{-\frac{1}{2}} 
e^{-2\sqrt{ (24X+2 \pi n+\varepsilon) \left (\frac{\pi}{12}+\varepsilon \right)}}.
\end{multline*}
\end{corollary}

\section{A summation formula for a Riesz means for holomorphic cusp forms}\label{section4}
In this section, we will use the theory of $L$-series of harmonic Maass forms to establish another summation formula applying to  holomorphic cusp forms. The difference from the summation formula of the last section is that it gives an expression for a finite Riesz-type sum, as the summation formula of \cite{BDGRT} did. However, in contrast to \cite{BDGRT}, the formula here is not based on the $L$-series of \cite{ShanSingh22} but on the more general $L$-series of \cite{DLRR23}. In particular, it originates in a different set of test functions which, in turn, works more naturally for non-negative weights $k$. By \eqref{genuine}, when $k \ge 0$, $H'_k(N,\chi)$ includes the elements of $M_k(N, \chi)$ and, in the case of weight $\frac{1}{2}, \frac{3}{2}$ and $2$ only, ``genuine" harmonic Maass forms (of polynomial growth). Here, we will consider only the elements of $M_k(N, \chi)$ for even $k$. 

The family of test functions with which we will work is given by 
$$\psi_s(t):=\chi_{(0, 1)}(t)(1-t)^{s-1}t^{k-s-1}/\Gamma(s), \, \, \text{where $\chi_A$ is the indicator function of the set $A.$}$$ 
We can then show the following 
\begin{proposition}\label{anaCont} Let $k \in 2\mathbb N$ and $f \in S_k(N,\chi)$ with $n$-th Fourier coefficient $c(n)$. The $L$-series $L_f(\psi_s)$ converges absolutely in the region $0<\sigma:=\Re(s)<(k-1)/2$, and has an analytic continuation to all $s\in\mathbb{C}$. For all $s \in \mathbb C$, it satisfies the functional equation 
\begin{equation}\label{FE1}
L_f(\psi_s) =i^kN^{1-\frac{k}{2}}L_{f|_k w_N}(\psi_s|_{2-k}w_N).
\end{equation}
\end{proposition} 
\begin{proof} First, with \cite[(13.4.1)]{DLMF}, followed by an application of \cite[(13.2.39)]{DLMF}, we obtain
$$\mathcal{L}\psi_s(2\pi n)=\frac{1}{\Gamma(s)}\int_0^1(1-r)^{s-1}r^{k-s-1}e^{-2\pi n r}dr=\frac{\Gamma(k-s)}{\Gamma(k)}\frac{1}{e^{2\pi n}}M(s, k,2\pi n),$$
where $M(a,b,z)$ is Kummer's confluent hypergeometric function. Since $|\Gamma(s)|(\mathcal{L}|\psi_s|)(2\pi n)=\Gamma(\sigma)(\mathcal{L}\psi_\sigma)(2\pi n),$  the asymptotics $M(a, b, x)=O(x^{a-b}e^{x})$ (\cite{DLMF} (13.7.1)) and $c(n)=O(n^{\frac{k-1}{2}+\varepsilon})$ imply that \eqref{Ff}
converges in the region under consideration and we have
\begin{equation}\label{lserieslhs}
L_{f}(\psi_s)=\frac{\Gamma(k-s)}{\Gamma(k)}\sum_{n\geq0}\frac{c(n)}{e^{2\pi n}}M(s,k, 2\pi n),
\end{equation}
To investigate the behaviour of $L_{f|_k w_N}(\psi_s|_{2-k}w_N)$, let the Fourier coefficients of $f|_kw_N$ be denoted by $d(n).$ We note that, for all $s\in\mathbb{C}$, $(\psi_s|_{2-k}w_N)(x)=\chi_{(1/N,\infty)}(x)(Nx-1)^{s-1}/\Gamma(s)$, and hence
\begin{equation}\label{LapU} \scrL(\psi_s|_kw_N)(y)=\frac{1}{N\Gamma(s)}\int_1^{\infty}(x-1)^{s-1}e^{\frac{-y}{N}x}dx=\frac{e^{\frac{-y}{N}}}{N} \left (\frac{y}{N} \right )^{-s},
\end{equation}
for all $y>0$. This shows that the holomorphic part of \eqref{Ff} for $f|_kw_N$ and $\psi|_{2-k}w_N$ converges for all $s \in \C$. With \eqref{LapU}, we have
\begin{align}\label{kummereqRHSholo}
L_{f|_{k}w_N}(\psi_s|_{2-k}w_N)=\frac{1}{N}
\sum_{n\geq0}d(n)e^{-\frac{2\pi n}{N}} \left (\frac{2 \pi n}{N} \right )^{-s}.
\end{align}
Therefore, for $0<\Re(s)<(k-1)/2$, $\psi_s \in \mathcal F_f$ and $\psi_s|_{2-k}w_N \in \mathcal F_{f|_k w_N}$, and thus, Theorem 4.5 of \cite{DLRR23} holds yielding \eqref{FE1} for all such $s$ in that range. 
However, we have also shown that $L_{f|_kw_N}(\psi_s|_{2-k}w_N)$ absolutely converges to an analytic function for \emph{all} $s \in \mathbb C$. Therefore, via \eqref{FE1}, $L_f(\psi_s)$ is analytically continued to the entire $s$-plane. 
\end{proof}

We can now state a summation formula for a cusp form, based on the above $L$-series.
\begin{theorem} Let $f \in S_k(N, \chi)$ with Fourier coefficients $c(n)$ and denote the Fourier coefficients of $f|_kw_N$ by $d(n).$ Then, for $\rho> 1,$ and each $X>0$ we have
\begin{multline}\label{SFS3}
\sum_{n<X/(2 \pi)}\frac{d(n)e^{-2 \pi n}}{(2\pi n)^{\frac{k-1}{2}}} \left (1-\frac{2 \pi n}{X} \right )^{\rho}= i^{-k}N^{\frac{k}{2}-1} L\left (\psi_{\frac{k-1}{2}} \right )
\\+\frac{\Gamma(\rho+1)N^{\frac{k}{2}-1}}{2 \pi i^{k+1}\Gamma(k)} \sum_{n>0}\frac{c(n)}{e^{2 \pi n}}
 \int_{(-\varepsilon)}X^{s}\frac{\Gamma \left (\frac{k+1}{2}-s \right )\Gamma(s)}{\Gamma(s+\rho+1)}M \left (\frac{k-1}{2}+s, k, 2 \pi n \right ) ds.
\end{multline}
\end{theorem}
\begin{proof} 
For $X>1$, $\varepsilon>0,$ $\rho>0$ 
we consider the integral
\begin{equation}
    \label{Idef}
I=\int_{(-\varepsilon)} X^s\frac{\Gamma(s)}{\Gamma(s+\rho+1)}L_f\left (\psi_{\frac{k-1}{2}+s} \right )ds.
\end{equation}
By the functional equation \eqref{FE1}, we can express $L_f(\psi_{\frac{k-1}{2}+s})$ as in \eqref{kummereqRHSholo} and we see that
\begin{equation}\label{abscon}\sum_{n>0}\int_{(\pm \varepsilon)} \left | X^s \frac{\Gamma(s)}{\Gamma(s+\rho+1)} \right |
\frac{e^{-2\pi n}|d(n)|}{(2 \pi n)^{\frac{k-1}{2} \pm \varepsilon}}ds
\ll X^{\pm \varepsilon} \sum_{n>0}\int_{\mathbb R} (1+|t|)^{-\rho-1}e^{-2 \pi n} n^{\varepsilon_1}dt< \infty.
\end{equation}
Here we used the Stirling formula and the Ramanujan bound for $d(n)$. Therefore, after an application of \eqref{FE1}, we can shift the line of integration to $\varepsilon$, picking up a pole at $s=0$, to obtain
\begin{multline} I=i^kN^{1-\frac{k}{2}}\int_{(-\varepsilon)} X^s\frac{\Gamma(s)}{\Gamma(s+\rho+1)}L_{f|_kw_N}\left (\psi_{\frac{k-1}{2}+s}|_{2-k}w_N \right )ds
\\=i^kN^{1-\frac{k}{2}}\int_{(\varepsilon)} X^s\frac{\Gamma(s)}{\Gamma(s+\rho+1)}L_{f|_k w_N}\left (\psi_{\frac{k-1}{2}+s}|_{2-k} w_N \right )ds-\frac{2 \pi i}{\Gamma(\rho+1)}L_{f}\left (\psi_{\frac{k-1}{2}} \right ).
\end{multline}
Because of \eqref{abscon} we can interchange the order of summation in the last integral, so that with \eqref{kummereqRHSholo} and 7.3(20) of \cite{erd} we get
$$\int_{(\varepsilon)} \frac{X^s\Gamma(s)}{\Gamma(s+\rho+1)}L_{f|_kw_N}\left (\psi_{\frac{k-1}{2}+s}|_{2-k}w_N \right )ds=\frac{2 \pi i}{\Gamma(\rho+1)}\sum_{n<X/(2 \pi)}\frac{d(n)e^{-2 \pi n}}{(2\pi n)^{\frac{k-1}{2}}} \left (1-\frac{2 \pi n}{X} \right )^{\rho}.
$$
On the other hand, by Prop. \ref{anaCont}, $(k-1)/2+s$, with $\Re(s)<0$, is in the region of absolute convergence of the $L$-series \eqref{lserieslhs} and thus, 
once the interchange of summation and integration is justified, equation \eqref{Idef} gives
\begin{equation}\label{I2}I=\frac{1}{\Gamma(k)} \sum_{n>0}\frac{c(n)}{e^{2 \pi n}}
 \int_{(-\varepsilon)}X^s\frac{\Gamma \left (\frac{k+1}{2}-s \right )\Gamma(s)}{\Gamma(s+\rho+1)}M \left (\frac{k-1}{2}+s, k, 2 \pi n \right ) ds.
\end{equation}

Therefore, to complete the proof of the theorem, we only need to show that the interchange of summation and integration is justified. To this end we will need the following uniform bound for $M(a, b, x)$.
\begin{lemma}\label{ubound} For $a>0,$ $b>2a+\frac{1}{2}$ and $x \in \mathbb R$, we have 
$$M\left( a+ ix,\, b,\, y \right) \ll_{a, b, \varepsilon} \frac{e^y y^{a-b+\varepsilon}}{\left | \Gamma\left( a + ix \right) \right |} \quad \text{as $y \to \infty.$}$$
\end{lemma}
{\bf Proof of Lemma.}
By (13.14.4), followed by (13.16.4) of \cite{DLMF} we have, for $y>0,$
\begin{equation*}
M\left( a + ix,\, b,\, y \right) =\frac{ e^{\frac{y}{2}}}{y^{\frac{b}{2}}} M_{\frac{b}{2}-a- ix,\, \frac{b-1}{2}} (y)
=\frac{y^{\frac{1-b}{2}} \Gamma(b)}{\Gamma\left(a + ix \right)}
\int_0^{\infty}e^{-t}t^{a-\frac{b+1}{2}+ix} I_{b-1}\left(2\sqrt{y t}\right)dt
\end{equation*}
Set
$$M_1:=
\frac{y^{\frac{1-b}{2}} \Gamma(b)}{\Gamma\left(a + ix \right)}
\int_0^{1}e^{-t}t^{a-\frac{b+1}{2}+ix} I_{b-1}\left(2\sqrt{y t}\right)dt
$$
and $M_2$ is the corresponding term with integral limits $1$ and $\infty$. Then \cite[(10.40.1)]{DLMF} implies that
$$
M_2 \ll  \frac{y^{\frac{1-b}{2}}}{\left | \Gamma\left( a+ ix \right) \right |}  \int_1^{\infty} e^{-t} t^{a-\frac{b+1}{2}} \frac{e^{2\sqrt{yt}}}{(yt)^{\frac14}} dt=\frac{y^{\frac14-\frac{b}{2}}}{\left | \Gamma\left( a+ ix \right) \right |}  \int_1^{\infty} e^{-t+2\sqrt{yt}} t^{a-\frac{b}{2}-\frac{3}{4}} dt
$$
The change of variables $x= \sqrt{t/y}$ shows that
\begin{equation}\label{I0}
\int_1^{\infty} e^{-t+2\sqrt{yt}} t^{a-\frac{b}{2}-\frac{3}{4}} dt=
2e^y y^{a-\frac{b}{2}+\frac{1}{4}}\int_{\frac{1}{\sqrt{y}}}^{\infty} e^{-y (x-1)^2} x^{2a-b-\frac{1}{2}}dx.
\end{equation}
To bound the integral, we set, for convenience: 
$$
A = \int_{\frac{1}{r} }^{\infty} e^{-(x-1)^2 r^2} x^{2a-b-\frac{1}{2}} \, dx 
$$
where $r>1,$ and for an $\varepsilon_1>0$, we decompose as
\begin{equation}\label{I}
A= \left( \int_{1 + \frac{1}{r^{1-\varepsilon_1}}}^{\infty} + \int_{1 - \frac{1}{r^{1-\varepsilon_1}}}^{1 + \frac{1}{r^{1-\varepsilon_1}}} + \int_{\frac{1}{r}}^{1 - \frac{1}{r^{1-\varepsilon_1}}} \right) e^{-(x-1)^2 r^2} x^{2a-b-\frac{1}{2}} \, dx 
\end{equation}
For the first and third integral, we note that, since then $|x - 1| > r^{\varepsilon_1-1}$,
$$
e^{-(x-1)^2 r^2} x^{2a-b-\frac{1}{2}} < e^{-\frac{r^2}{r^{2-2\varepsilon_1}}} x^{2a-b-\frac{1}{2}}
=e^{-r^{2\varepsilon_1}} x^{2a-b-\frac{1}{2}}
$$
and hence, the first integral of \eqref{I} is:
\begin{equation}\label{I1}
\le \int_{1 + \frac{1}{r^{1-\varepsilon_1}}}^{\infty} e^{-r^{2\varepsilon_1}} x^{2a-b-\frac{1}{2}} dx = e^{-r^{2\varepsilon_1}} \left ( \int_{1 + \frac{1}{r^{1-\varepsilon_1}}}^{\infty} x^{2a-b-\frac{1}{2}} dx \right ) =O(e^{-r^{2\varepsilon_1}})
\end{equation}
Similarly, the third integral of \eqref{I} is $O(e^{-r^{2\varepsilon_1}})$. 
For the second integral we have:
\begin{multline*}
 \int_{1 - \frac{1}{r^{1-\varepsilon_1}}}^{1 + \frac{1}{r^{1-\varepsilon_1}}} e^{-(x-1)^2 r^2} x^{2a-b-\frac{1}{2}} \, dx  \le  \int_{1 - \frac{1}{r^{1-\varepsilon_1}}}^{1 + \frac{1}{r^{1-\varepsilon_1}}}  x^{2a-b-\frac{1}{2}} \, dx
 \\
 =  \frac{ (1 + \frac{1}{r^{1-\varepsilon_1}})^{2a-b+\frac{1}{2}} - (1 - \frac{1}{r^{1-\varepsilon_1}})^{2a-b+\frac{1}{2}}}{2a-b+\frac{1}{2}}.
\end{multline*}
By the Taylor expansion of $(1+x)^{2a-b+\frac{1}{2}}$ at $x=0,$ we deduce that this is $O(r^{-1+\varepsilon_1}).$
Comparing with \eqref{I1}, we see that this bound dominates as $r \to \infty.$ Therefore, $A=O(r^{-1+\varepsilon_1})$, which we can apply to \eqref{I0} with $r=\sqrt{y}$, to deduce that
\begin{equation}\label{M2b}
M_2 \ll
\frac{y^{\frac14-\frac{b}{2}}e^{y } y^{a-\frac{b}{2}+\frac14}y^{-\frac12+\frac12 \varepsilon_1}}{\left | \Gamma\left( a + ix \right) \right |}=\frac{e^y y^{a-b+\frac12 \varepsilon_1}}{\left | \Gamma\left( a + ix \right) \right |}.
\end{equation}
To bound $M_1$, we note that, by the series definitions of $I_\nu$ and $J_{\nu}$, we have $|I_{b-1}\left(2\sqrt{2\pi nt}\right)|=|J_{b-1}\left(2i\sqrt{2\pi nt}\right)|$. This, combined with inequality $|J_{\nu}(z)| \le |z/2|^{\nu}e^{|\Im(z)|}/\Gamma(\nu+1)$, valid for $\nu \ge -1/2$ (\cite[(10.14.4)]{DLMF}) implies that
$$M_1 \ll \frac{y^{\frac{1-b}{2}}}{|\Gamma\left(a + ix \right)|}
\int_0^{1}e^{-t}t^{a-\frac{b+1}{2}} (yt)^{\frac{b-1}{2}} e^{2 \sqrt{yt}}dt\le
\frac{ e^{2\sqrt{y}}}{|\Gamma\left(a + ix \right)|}
\int_0^{1}e^{-t} t^{a-1}dt.
$$
Comparing this with \eqref{M2b} we deduce the lemma.
\qed

We can apply this lemma with $y=2 \pi n$, $a=\frac{k-1}{2}-\varepsilon$ and $b=k$ to get, with 
$c(n) \ll n^{\frac{k-1}{2}+\varepsilon_2}$,
\begin{multline}\label{|I21|} 
\sum_{n>0}\frac{|c(n)|}{e^{2 \pi n}}\int_{\mathbb R} \left |  \frac{\Gamma \left (\frac{k+1}{2}+\varepsilon-ix \right )\Gamma(-\varepsilon+ix)}{\Gamma(-\varepsilon+\rho+1+ix)} \right | \left | M \left (\frac{k-1}{2}-\varepsilon+ix, k, 2 \pi n \right ) \right |dx 
\\
\ll \sum_{n>0}\frac{n^{\frac{k-1}{2}+\varepsilon_2-\frac{k}{2}-\frac12-\varepsilon+\varepsilon_1} e^{2 \pi n}}{e^{2 \pi n}}
 \int_{\mathbb R} \left |  \frac{\Gamma \left (\frac{k+1}{2}+\varepsilon-ix \right )\Gamma(-\varepsilon+ix)}{\Gamma(-\varepsilon+\rho+1+ix) \Gamma\left( \frac{k-1}{2} - \varepsilon + ix \right)}
 \right | dx.
\end{multline}
For $\varepsilon_1, \varepsilon_2$ such that $\varepsilon_2+\varepsilon_1<\varepsilon$  the series converges and, by Stirling, 
$$\left |  \frac{\Gamma \left (\frac{k+1}{2}+\varepsilon-ix \right )\Gamma(-\varepsilon+ix)}{\Gamma(-\varepsilon+\rho+1+ix) \Gamma\left( \frac{k-1}{2} - \varepsilon + ix \right)}
 \right |  \ll (1+|x|)^{2\varepsilon-\rho}.$$ For $\rho>1$, this shows that \eqref{|I21|} converges as required.
\end{proof}
\begin{corollary}
    With the assumptions of the theorem, we have, for $\rho> 1$ and each $X>0$,
\begin{equation}
\sum_{n<X/(2 \pi)}\frac{d(n)e^{-2 \pi n}}{(2\pi n)^{\frac{k-1}{2}}} \left (1-\frac{2 \pi n}{X} \right )^{\rho}= i^{-k}N^{\frac{k}{2}-1} L\left (\psi_{\frac{k-1}{2}} \right )+O(X^{-\varepsilon}).
\end{equation}
\end{corollary}
\begin{proof}
This follows once we apply Lemma \ref{ubound} to the last term of \eqref{SFS3} as we did in \ref{|I21|}.
\end{proof}

\bibliographystyle{alpha}
\bibliography{references.bib}

\newcommand{\etalchar}[1]{$^{#1}$}
\begin{thebibliography}{EMOT54}

\bibitem[BDG{\etalchar{+}}25]{BDGRT}
O.~Beckwith, N.~Diamantis, R.~Gupta, L.~Rolen, and K~Thalagoda.
\newblock Summation formulas for hurwitz class numbers and other mock modular coefficients.
\newblock {\em (submitted)}, 2025.

\bibitem[BFOR17]{thebook}
Kathrin Bringmann, Amanda Folsom, Ken Ono, and Larry Rolen.
\newblock {\em Harmonic {M}aass forms and mock modular forms: theory and applications}, volume~64 of {\em American Mathematical Society Colloquium Publications}.
\newblock American Mathematical Society, Providence, RI, 2017.

\bibitem[DLRR23]{DLRR23}
Nikolaos Diamantis, Min Lee, Wissam Raji, and Larry Rolen.
\newblock {$L$}-series of harmonic {M}aass forms and a summation formula for harmonic lifts.
\newblock {\em Int. Math. Res. Not. IMRN}, (18):15729--15765, 2023.

\bibitem[EMOT54]{erd}
A.~Erd\'elyi, W.~Magnus, F.~Oberhettinger, and F.~G. Tricomi.
\newblock {\em Tables of integral transforms. {V}ol. {I}}.
\newblock McGraw-Hill Book Co., Inc., New York-Toronto-London, 1954.
\newblock Based, in part, on notes left by Harry Bateman.

\bibitem[OLBC10]{DLMF}
Frank W.~J. Olver, Daniel~W. Lozier, Ronald~F. Boisvert, and Charles~W. Clark, editors.
\newblock {\em N{IST} handbook of mathematical functions}.
\newblock U.S. Department of Commerce, National Institute of Standards and Technology, Washington, DC; Cambridge University Press, Cambridge, 2010.
\newblock With 1 CD-ROM (Windows, Macintosh and UNIX).

\bibitem[Ono04]{ono}
Ken Ono.
\newblock {\em The web of modularity: arithmetic of the coefficients of modular forms and {$q$}-series}, volume 102 of {\em CBMS Regional Conference Series in Mathematics}.
\newblock Conference Board of the Mathematical Sciences, Washington, DC; by the American Mathematical Society, Providence, RI, 2004.

\bibitem[SS22]{ShanSingh22}
Karam~Deo Shankhadhar and Ranveer~Kumar Singh.
\newblock An analogue of {W}eil's converse theorem for harmonic {M}aass forms of polynomial growth.
\newblock {\em Res. Number Theory}, 8(2):Paper No. 36, 27, 2022.

\end{thebibliography}

\end{document}